\documentclass{amsart}

\usepackage{amsmath,amsthm,amssymb,amsfonts,parskip, mathtools, yfonts}
\usepackage{tipa}
\usepackage[nobysame]{amsrefs}
\usepackage{color}
\usepackage{enumitem}
\usepackage[all]{xy}
\usepackage{endnotes}
\usepackage[mathscr]{eucal}

\def\mp{{\mathfrak p}}
\def\mm{{\mathfrak m}}

\theoremstyle{plain}
\newtheorem*{Theorem}{Theorem}
\newtheorem{theorem}{Theorem}[section]
\newtheorem{proposition}[theorem]{Proposition}
\newtheorem{lemma}[theorem]{Lemma}
\newtheorem*{Question}{Question}
\newtheorem{corollary}[theorem]{Corollary}
\newtheorem*{Corollary}{Corollary}

\theoremstyle{definition}
\newtheorem{definition}[theorem]{Definition}
\newtheorem{example}[theorem]{Example}

\newtheorem*{Conjecture}{Conjecture}

\theoremstyle{remark}
\newtheorem{remark}[theorem]{Remark}

\numberwithin{equation}{theorem}

\def\lra{\longrightarrow}

\DeclarePairedDelimiter{\brackets}{(}{)}
\newcommand{\T}[2]{\operatorname{{\tau}}_{{#2}} (#1)}
\newcommand{\End}[2]{\operatorname{End}_{#1} (#2)}

\renewcommand{\hom}[3]{\operatorname{Hom}_{#1} (#2, #3)}

\newcommand{\IM}{\operatorname{Im} \brackets}

\newcommand{\Ass}{\operatorname{Ass} \brackets}
\newcommand{\Supp}{\operatorname{Supp} \brackets}
\newcommand{\Min}{\operatorname{Min} \brackets}
\newcommand{\Ann}{\operatorname{Ann}_R }

\newcommand{\Ext}[4]{\operatorname{Ext}^{#1}_{#2} (#3, #4)}
\newcommand{\grade}{\operatorname{grade}}
\newcommand{\Spec}{\operatorname{Spec} \brackets}

\begin{document}

\title[]{Self-injective commutative rings \\ have no nontrivial rigid ideals}

\author{Haydee Lindo}

\curraddr{Dept. of Mathematics \& Statistics, Williams College, Williamstown, MA, USA}
\email{haydee.m.lindo@williams.edu}

\date{\today}



\maketitle

\begin{abstract}
We establish a link between trace modules and rigidity in modules over Noetherian rings. Using the theory of trace ideals  we make partial progress on a question of Dao, and on the Auslander-Reiten conjecture over Artinian Gorenstein rings.
\end{abstract}

\section{Introduction}

Let $R$ be a ring and $M$, $X$ $R$-modules. The \emph{trace module of $M$ in $X$}, denoted  $\T X M$,  is the $R$-module $\sum \alpha (M)$ as $\alpha $ ranges over $\hom R M X$. Such a trace module is \emph{proper} provided $\T X M \subsetneq X$. We identify a link between trace modules over $R$ and the existence of self-extensions of $R$-modules. The main result is

 \begin{Theorem} 
 Let $R$ be an local Artinian Gorenstein ring. If $M$ is a syzygy of a proper trace module then $\Ext 1 R M M \not= 0$.
 \end{Theorem}

Every proper ideal in an Aritinian Gorenstein ring is a proper trace module in $R$; see Proposition \ref{istrace}.  As a consequence, we obtain a positive answer for ideals to a question of Dao \cite{Dao1}, and we settle the Auslander-Reiten conjecture for ideals and their syzygies over Artinian Gorenstein rings.

 \begin{Corollary} 
Let $R$ be an local Artinian Gorenstein ring. If an $R$-module $M$ is rigid and appears as a positive or negative syzygy of an ideal then $M$ is free. In particular, the Auslander-Reiten conjecture holds for all ideals $I \subseteq R$. 
\end{Corollary}

Recall that we say a left $R$-module $M$ is \emph{rigid} if $\Ext 1 R M M = 0$ and that a commutative Noetherian ring is Artinian Gorenstein if and only if it is self-injective.

  In Section \ref{prelims} we present the needed results concerning trace modules and establish a key lemma, Lemma \ref{homtrace}. In Section \ref{rigid} we discuss some cases in which rigidity implies projectivity and prove the main result, Theorem \ref{ARCAGor}.

\section{Trace Modules and Trace Ideals} \label{prelims}
Let $R$ be a Noetherian ring.  In this section $R$ need not be commutative. We recall the properties of trace modules needed for subsequent sections.  
%changed from 8.1

Given a  left $R$-module $M$, we write $M^*$ for $\hom R M R$.

\begin{definition}
Let $ M$ and $X$ be left $R$-modules. The \emph{trace module of $M$ in $X$ } is
\[ \T X M := \sum_{\alpha \in \hom R M X} \alpha(M).\]
%corrected from 8.1

We say an $R$-module $M$ is a \emph{trace module} in $X$ provided $M = \T X A \subseteq X$ for some $R$-modules $A$ and $X$. We call such an $M$ a \emph{proper trace module} when containment is strict and a \emph{trace ideal} when $X=R$. 

\end{definition}

\begin{definition}\label{gen}
Let $M$ and $X$ be $R$-modules. We say $X$ is \emph{generated} by $M$ if  $\T X M = X$, that is, $X$ is the homomorphic image of a direct sum of copies of $M$. 
\end{definition}

\begin{lemma} \label{chartrace}
Consider $R$-modules $M\subseteq X$. The following are equivalent:

 \begin{enumerate}[label=(\roman*)]
 \item $M$ is a trace module in $X$;
 \item $M = \T X M$;
 \item $\End R M = \hom R M X$, that is,  every homomorphism from $M$ to $X$ has its image in $M$.
 \end{enumerate}
\end{lemma}

\begin{proof}
$(i\Rightarrow ii)$: Evidently $M \subseteq \T X M$. If $M =\T X A $ for some $R$-module $A$, then $A$ generates $M$ and $M$ generates $\T X M$. It follows that $\T X M \subseteq \T X A = M$.

$(ii\Rightarrow iii)$: One has $\End R M \subseteq \hom R M X$. Now, given  $\alpha$ in $\hom R M X$, by definition $\IM \alpha \subseteq \T X M = M$. Therefore $\hom R M X \subseteq \End R M$.

$(iii\Rightarrow i)$: Given any  $\alpha$ in $\hom R M X$, then $\End R M = \hom R M X$ implies that $\IM \alpha \subseteq M$. It follows that $\T X M \subseteq M$ and therefore $\T X M= M$. \end{proof}

\begin{remark}
By Lemma \ref{chartrace}, to say $M$ is a proper trace module in $X$ is to say  $M \subsetneq X$ and $M = \T X M$. Note,  all modules are trivially trace modules since $M = \T M M = \T M R$. It is left to characterize proper trace modules and trace modules of modules that do not generate the entire category of $R$-modules. 
\end{remark}

\begin{example}
Consider the ring $R= k[[x^3, x^4, x^5]]$,  where $k$ is a field. Note that $R$ is a one-dimensional Cohen-Macualay domain which is not Gorenstein  and the maximal ideal of $R$ is $\mathfrak{m} = (x^3, x^4, x^5)$. The set of trace ideals in $R$ is $\{ (0), \mathfrak{m}, R\}$; see \cite[Example 30]{LucasRTP}
\end{example}

\begin{remark}
Auslander and Green \cite{tracequotient} identify several classes of proper trace modules. Let  $\Lambda$ be an Artin algebra. If $X$ is a $\Lambda$-module, then the socle of $X$ is a trace module in $X$; see the proof of \cite[Proposition 2.5]{tracequotient}. If  $C$ is a non-rigid $\Lambda$-module such that $\End {\Lambda} C$ is a division ring, then $C$ is a proper trace module in some $\Lambda$-module $X$; see \cite[Proposition 2.3]{tracequotient}. Also, if $M\subseteq X$  is a \emph{waist} in $X$ then $M$ is a trace module in $X$, that is, if for any $N \subseteq X$ either $N \subseteq M$ or $M \subseteq N$; see \cite[Proposition 2.1]{tracequotient}. 

See Proposition \ref{istrace} for further examples of  trace ideals. 
\end{remark}

\begin{lemma} \label{homtrace} \label{rigidtrace}
Suppose $M$ is a proper submodule of $X$. If $\hom R M {X/M} = 0$ then $M$ is a trace module in $X$. The converse holds when $M$ is also rigid. 
\end{lemma}
\begin{proof}
Applying $\hom R M -$ to the exact sequence\[ 0 \lra M \overset{}{\lra} X \lra X/M \lra 0.\]

 yields an exact sequence 
\[ 0 \lra \End R M \lra \hom R M X \lra \hom R M {X/M} \lra \Ext 1 R M M.\]

If $\hom R M {X/M} = 0$ then $\End R M = \hom R M X$ and  $M$ is a trace module by Lemma \ref{chartrace}. On the other hand, if $M$ is a rigid trace module in $X$, then $\hom R M {X/M} = 0$ by the exactness of the sequence. \end{proof}

\begin{remark}
One may use Lemma \ref{homtrace} to show that certain proper trace modules cannot be rigid, and also to show that certain rigid trace modules $M$ in $X$ cannot be proper, that is, $M=X$; see, for example, Theorem \ref{ARCAGor} and Proposition \ref{rigiditythm} respectively. 
Note, the converse to the statement in Lemma \ref{homtrace} does not hold when $M$ is not rigid. Section \ref{rigid} focuses on trace ideals such that $\hom R I {R/I} \not=0 $.

\end{remark}

The remainder of this paper considers the consequences of Lemma \ref{homtrace} over commutative Noetherian rings. 

\begin{remark}
Suppose $R$ is commutative. By Hom-Tensor adjunction the hypothesis $\hom R I {R/I} =0 $ is equivalent to $(I/I^2)^* = 0$, where the dual is taken with respect to $R/I$. Lemma \ref{homtrace} proves that $(I/I^2)^* = 0$ implies $I$ is a trace ideal; see Example \ref{zerohom} (b). 

The $R/I$-module $I/I^2$ is called the conormal module of $I$. It is well-known that the conormal module of $I$ can detect if $R/I$ is a complete intersection; see \cite{Vascidealsgen}. Here we see that the conormal module can also detect when $I$ is a trace ideal.
  
\end{remark}

\begin{remark}

Consider the relationship between these three conditions:

 \begin{enumerate}[label=(\roman*)]
 \item $M$ is a  proper trace module in $X$,
 \item  $\hom R M {X/M} \not=0$, \mbox{ and}
 \item $M$ is rigid.
 \end{enumerate}

Lemma \ref{homtrace}  shows that a pair of modules $(M,X)$ cannot have all three. However, Example \ref{zerohom} shows that pairs $(M, X)$ may have any  two of these three properties. 

\end{remark}
\begin{example} \label{zerohom}
\
\begin{itemize}
\item[(a)] Any proper free ideal, $I$, is a rigid ideal and $\hom R I {R/I} \not=0$. For any such $I$ one has $\T R I = R$. 

\item[(b)]\label{conormal0} Let $R = k[[x,y]]/(xy)$. Then $I = (y)$ is a rigid trace ideal which is not free. Here $\hom R I {R/I} =0$ because $\Ann I = (x)$ consist of nonzerodivisors on $R/I \cong k[[x]]$; see Lemma \ref{24} and \cite[Example 1.2]{HWrigidityoftor}.

\item[(c)] Let $R = k[[x,y]]/(xy)$. Then $J = (x,y)$ is a trace ideal with $\hom R J {R/J}\not=0$, that is not rigid and not free. Whenever $R$ is a local ring which is not a DVR, its maximal ideal will be an example of this kind.

\end{itemize}
\end{example}

\section{Rigid Trace Modules over Gorenstein Rings}\label{rigid}

 In this section $R$ is a commutative Noetherian ring. We use Lemma \ref{homtrace} to establish theorems concerning rigid modules.  To that end, we first identify of examples of trace ideals and pairs of $R$-modules $M \subseteq X$ for which $\hom R M {X/M} \not=0$.

 The proof of the result below,  is an application of the proof of Example 2.4 in \cite{Lindo1}. 

\begin{proposition}\label{istrace}
If $\Ext 1 R {R/I} R = 0$ then $I$ is a trace ideal. 
\end{proposition}

\begin{proof}
Applying $\hom R - R$ to the sequence \[ 0 {\lra} I \overset{i}{\lra} R \lra R/I \lra 0, \] 
 one gets the exact sequence 
\[0 \lra \hom R {R/I} R \lra \hom R R R \overset{i^*}{\lra} \hom R I R \lra \Ext 1 R {R/I} R \lra \cdots .\]

Because $\Ext 1 R {R/I} R = 0$, the map $i^*$ is  surjective and each $R$-homomorphism from $I$ to $R$ is given by multiplication by an element of $R$. Therefore $\hom R I R = \End R I$  and  $I$ is a trace ideal by Lemma \ref{chartrace}.
\end{proof}

\begin{remark} \label{extraceideals}
When $R$ is a commutative Noetherian  ring
\begin{align*}
\grade I &= \mbox{min}\{ i | \Ext i R {R/I} R \not= 0\}, \end{align*}
 by \cite[Theorem 1.2.5]{BandH}. Also, $R$ is Artinian Gorenstein if and only if it is self-injective. It follows from Proposition \ref{istrace} that  $I$ is a trace ideal if

 \begin{enumerate}[label=(\roman*)]
 \item $\grade I \geq 2$ \mbox{ or}
 \item $R$ is an Artinian Gorenstein ring.

 \end{enumerate}
 In fact, a Noetherian ring $R$ is Artinian Gorenstein if and only if every ideal is a trace ideal; see \cite{LLP}. 
However, not every grade zero ideal in a Gorenstein ring is a trace ideal; see Example \ref{hypersurface}. Nevertheless, if $\Min I \cap \Supp I \not= \varnothing$, where $\Min I$ is the set of prime ideals minimally containing $I$, then a standard reduction to the Artinian case still precludes rigidity in proper grade zero ideals in Gorenstein rings; see Proposition \ref{rigiditythm}.

\end{remark}

\begin{example}\label{hypersurface}
Let $R = \mathbb{Q}[x,y]/(x^2y^2)$. For the grade zero ideal $I = (x^5, xy^7)$, $\T I R = (x^2, xy^2)$. To see this, recall that $\T I R$ is the ideal generated by the entries of the left kernel of the presentation matrix of $I$; see \cite[Remark 3.3]{Vascaffine}. 
\end{example}

\begin{definition}
We say $R$ is \emph{generically Gorenstein} provided $R_{\mp}$ is Gorenstein for each $\mp$ minimal in $\Spec R$.
\end{definition}

\begin{proposition} \label{rigiditythm}
 Let $R$ be a local ring that is generically Gorenstein and let $I \subsetneq R$ be an ideal. Suppose $\grade I = 0$ and $\Min I \cap \Supp I \not= \varnothing$ or $\grade I > 1$. Then $I$ is  not rigid. 
 \end{proposition}
 \begin{proof}
 Assume that $\grade I = 0$. Pick $\mp \in \Min I \cap \Supp I$, then $R_{\mp}$ is Artinian Gorenstein.
 
$I_{\mp}$ is a trace ideal in $R_{\mp}$; see Proposition \ref{istrace}. The ideal $I_{\mp}$ is $\mp R_{\mp}$-primary therefore $\hom {R_{\mp}} {I_{\mp}} {{R_{\mp}}/{I_{\mp}}} \not=0$ by Remark \ref{set} and   \[ 0 = {\Ext 1 R I I}_{\mp} = {\Ext 1 {R_{\mp}} {I_{\mp}} {I_{\mp}}}.\] By Lemma \ref{rigidtrace}, $I_{\mp} =  R_{\mp}$. This contradicts the containment $I \subseteq \mp$. It follows that $\grade I \neq 0$. 
 
Now assume $\grade I \geq 2$. If $I\not= R$, $I$ is a rigid proper trace ideal such that $\hom R I {R/I} \not=0$; see Remark \ref{extraceideals}, Lemma \ref{24} and Remark  \ref{set}. This contradicts Lemma \ref{rigidtrace}, hence $I=R$. 
 \end{proof}

The following result is well-known; see, for example, \cite[Lemma 8.1]{24hrs}. 

\begin{lemma}\label{24}
Let $R$ be a commutative Noetherian ring and let $M$ and $ N$ be finitely generated $R$-modules, such that $M \otimes_R N \not= 0$. Then $\hom R M N =0$ if and only if  $\Ann M$ contains an $N$-regular element. 
\end{lemma}

\begin{remark} \label{set}
When $R$ is local and $M$ and $N$ are nonzero finitely generated $R$-modules $M\otimes_R N \not= 0$. Therefore, over a local ring, the hypothesis $\hom R M {X/M}\not=0$ in  Lemma \ref{homtrace}   is a mild condition equivalent to \[\Supp M \cap \Ass {X/M} = \Ass {\hom R M {X/M}}  \not=\varnothing.\] Among other pairs of modules, it is held by

 \begin{enumerate}[label=(\roman*)]
 \item  $(I,R)$ for proper $\mm$-primary ideals $I \subseteq R$;
  \item $(I,R)$ for all proper ideals $I \subseteq R$ of positive grade; 
 \item The $R$-modules $(M, X)$ whenever $X/M$ is nonzero and has finite length.
 \end{enumerate}

\end{remark}

  \begin{remark}
Over Gorenstein rings, rigidity passes from an MCM module to its (negative and positive) syzygies; see \cite[Proposition 7.3 (1)(i)]{Takahashi2006}  and \cite[Proposition 7]{ArayaARC}. This allows us to provide a partial answer to Question 9.1.4 in \cite{Dao1}:

 \end{remark}
 
 \begin{Question}
Let $R$ be a commutative Artinian, Gorenstein local ring and $M$ be a finitely generated $R$-module. If $M$ is rigid is $M$ free? 
\end{Question}

 \begin{theorem} \label{ARCAGor}
 Let $R$ be an local Artinian Gorenstein ring. If $M$ is a nonzero syzygy of a proper trace module, then $\Ext 1 R M M \not= 0$. In particular, if $M = \Omega^n I$ for some nonzero ideal $I \subset R$ and  $n \in \mathbb{Z}$, then $M$ is not rigid. 
 \end{theorem}
 
 \begin{proof}
Suppose $M$ is a syzygy of a proper trace module $N$ in $X$. Since $R$ is Artinian, the module $X/N$ has finite length. It follows that $\hom R N {X/N} \not=0$; see Remark \ref{set}. By Lemma \ref{rigidtrace}, $N$ is not rigid. Since rigidity passes to syzygies, it follows that $M$ is not rigid. 
%changed from v 8.1
 \end{proof}

\begin{Conjecture}[Auslander-Reiten conjecture]
Let A be an Artin algebra and $M$ a finitely generated $A$-module.
If $\Ext i A M {M \oplus A} =0$ for all $i >0$ then M is projective.
\end{Conjecture}

\begin{corollary} \label{ARCgor}
Let $R$ be an local Artinian Gorenstein ring. If an $R$-module $M$ is rigid and appears as a positive or negative syzygy of an ideal, then $M$ is free. In particular, the Auslander-Reiten conjecture holds for all ideals $I \subseteq R$. 
\end{corollary}

\begin{proof}
%Every cyclic module is a syzygy of an ideal and 

Every ideal $I$ in an Artinian Gorenstein ring is a trace ideal; see Proposition \ref{istrace}. 
By Theorem \ref{ARCAGor},  a rigid $I$ cannot be a proper trace ideal.  It follows that $I=R$ and that its syzygies are free $R$-modules. \end{proof}
%changed from 8.1
 
 \begin{remark}
 It is left to determine the full set of modules that can be realized as the syzygies of proper trace modules over Artinian Gorenstein rings. \end{remark}

\section*{Acknowledgements}

Thanks to  Srikanth Iyengar for detailed feedback on multiple drafts of this paper. Thanks to Andrew Bydlon, Jake Levinson and Adam Boocher for many key conversations. 

\bibliographystyle{amsplain}	% (uses file "plain.bst")
%\bibliography{myrefs}		% expects file "myrefs.bib"

\end{document}